\documentclass[11pt,reqno]{amsart}
\usepackage{blankersbasic}

\title[]{Extremality of Rational Tails Boundary Strata in $\Moduli_{g,n}$}

\author{Vance Blankers}
\address{Department of Mathematics, Northeastern University, Boston, Massachusetts 02115-5005}
\email{\href{mailto:v.blankers@northeastern.edu }{v.blankers@northeastern.edu }}
\thanks{The author was supported by the NSF grant DMS--1645877.}

\begin{document}
\allowdisplaybreaks
\setcounter{tocdepth}{1}

\maketitle

\begin{abstract} 
We review and develop some techniques used to investigate the effective cones of higher codimension classes. Our results show that a large collection of boundary strata of rational tails type are extremal in their effective cones on $\Moduli_{g,n}$ and provide evidence for the conjecture that all boundary strata of $\Moduli_{g,n}$ are extremal. As a corollary, we show that all boundary strata are extremal in genus zero.
\end{abstract}

\section{Introduction}
\label{sec:intro}

The cone of effective divisors on a projective variety $X$ dictates its birational geometry. When $X$ is a moduli space, birational models of $X$ often have new modular interpretations and useful connections to each other. For this reason and others,
the structure of the cone of effective divisors of $\Moduli_{g,n}$ has attracted a great deal of attention; see for example
\cite{chencoskun2014,castravettevelev,opie2016,mullane2017,mullane2020}.
More generally, there has been interest in probing the finer aspects of the birational geometry of moduli spaces by studying the cones of effective higher codimension cycles, e.g., \cite{mullane2017higher,chensurvey,mullane2019,blankers17}.

Cones of higher codimension cycles are significantly more difficult to understand than cones of divisors, in part because their positivity properties do not behave as well as those of divisors; for example, nef cycles may fail to be pseudoeffective in higher codimension \cite{delv11}. Moreover, some of the tools used to study the cone of effective divisors are not known to generalize (or are known to not generalize). For more on positivity of higher codimension cycles, see \cite{fulgerlehmann2014}.

There are two complementary types of results in studying extremality of cycles. The first is to show whether or not a given cone is rational polyhedral, or more specifically, whether there are infinitely-many extremal rays for the cone. In codimension one, this focus stems largely from interest in running the minimal model program. The second is to study a particular family of classes of more general interest that live in a family of cones and to establish which classes in the family are extremal in their respective cones.

This paper is of the latter type: we extend the work of \cite{chencoskun2015} and show that a large collection of boundary strata of rational tails type span extremal rays of the cones of effective classes of $\Moduli_{g,n}$. In order to do so, we first review some (pseudo)effective techniques for determining when (pseudo)effective classes are rigid and extremal in a (pseudo)effective cone. Our results present evidence for the following conjecture.

\begin{conjecture}
\label{conj:extremalboundary}
All boundary strata in $\Moduli_{g,n}$ are extremal.
\end{conjecture}

The paper is organized as follows: in Section \ref{sec:cones}, we review properties of effective cycles and cones and prove two important lemmas. In Section \ref{sec:curves} we cover basic information about $\Moduli_{g,n}$ and Hassett spaces, and lay the groundwork for the main results. Section \ref{sec:rattail} contains the main inductive results and some of their immediate applications.

\section{Background on Effective Cones}
\label{sec:cones}

In this section we review some standard definitions and results about effective cones of cycles on a variety and present some criteria for establishing extremality of effective cycles. Throughout, we assume all varieties are defined over $\mathbb{C}$ and all coefficients are taken to be $\mathbb{R}$-valued.

A \emphbf{cycle} on a complete projective variety $X$ is a finite formal sum of subvarieties of $X$; if all subvarieties in the sum are of dimensional $d$ (resp. codimension $k$), the cycle is \emphbf{$d$-dimensional} (resp. \emphbf{$k$-codimensional}). Two $d$-dimensional cycles $Z$ and $Z'$ on $X$ are \emphbf{numerically equivalent} if, for any polynomial $P$ of weight $d$ in Chern classes of vector bundles on $X$,
\begin{align*}
\int_X Z\cap P = \int_X Z' \cap P,
\end{align*}
where $\cap$ is the cap product (see \cite[Chapter 19]{fultoninttheorybook}). If $X$ is non-singular, numerical equivalence is equivalent to requiring the intersection product
\begin{align*}
\int_X Z \cdot V = \int_X Z' \cdot V
\end{align*}
for all subvarieties $V\subset X$ of codimension-$d$. Although $\Moduli_{g,n}$ -- the focus of this paper -- is not non-singular, the compatibility still holds for moduli spaces of curves by \cite{edidin1992}, as $\Moduli_{g,n}$ is $\mathbb{Q}$-factorial.

We denote by $[Z]$ the numerical equivalence class of $Z$ in $X$ and write $[Z] = [Z']$ if $Z$ and $Z'$ are numerically equivalent. Let $N_d(X)$ be the (finite-dimensional) $\mathbb{R}$-vector space of cycles of dimension $d$ modulo numerical equivalence, and let $N^k(X)$ be the vector space for cycles of codimension $k$. We caution that these spaces are in general not dual when $(d,k) \neq (1,1)$, but they are dual when $X$ and $Y$ are smooth or $\mathbb{Q}$-factorial (which is the focus throughout). The decision regarding which notation to use is based on whether it is more convenient to note the dimension or codimension of a given class.

A cycle is \emphbf{effective} if all of the coefficients in its sum are non-negative or if it is numerically equivalent to such a sum. The sum of two effective dimension-$d$ (resp. codimension-$k$) classes is again effective, as is any $\mathbb{R}_+$-multiple of the same, which gives a natural convex cone structure on the set of effective classes of dimension $d$ (resp. codimension $k$) inside $N_d(X)$ (resp. $N^k(X)$), called the \emphbf{effective cone} and denoted $\text{Eff}_d(X)$ (resp. $\text{Eff}^k(X)$).

\begin{definition}\label{def:effectivedecomp}
Let $X$ be a projective variety and $\alpha \in \text{Eff}_d(X)$. An \emphbf{effective decomposition of $\alpha$} is an equality
\begin{align*}
\alpha = \sum_{s=1}^{r} a_s [E_s] \in N_d(X),
\end{align*}
with each $a_s > 0$ and $E_s$ irreducible subvarieties in $X$ of dimension $d$.
\end{definition}

\begin{definition}\label{def:extremalrigid}
A class $\alpha\in\text{Eff}_d(X)$ is \emphbf{extremal} if for any effective decomposition of $\alpha$ all classes in the decomposition are proportional to $\alpha$; the class $\alpha$ is \emphbf{rigid} if any effective cycle with class $m\alpha$ is supported on the support of $\alpha$.
\end{definition}

There are numerical techniques to determine whether divisors are extremal; in particular, the following criterion is a powerful tool.

\begin{lemma}[{{\cite[Lemma 4.1]{chencoskun2014}}}]\label{lem:divisoronly}
Let $D$ be an irreducible effective divisor in a projective variety $X$, and suppose that $\mathcal{C}$ is a moving curve in $D$ satisfying $\DD [D]\cdot[\mathcal{C}] < 0$. Then $[D]$ is rigid and extremal in $X$. \hfill $\square$
\end{lemma} 

Unfortunately, there is no analog to Lemma \ref{lem:divisoronly} for higher codimension cycles, and its failure to generalize is partially responsible for the overall lack of information about cones of such classes. In these cases, subtler techniques must be used. Of particular utility is the notion of the \emphbf{index} of a cycle under a morphism.

\begin{definition}[{\cite{chencoskun2015}}]
Let $f:X\to Y$ be a morphism between complete varieties. For an irreducible subvariety $Z\subset X$ define the \emphbf{index of $Z$ under $f$} as
\begin{align*}
e_f(Z) = \dim Z - \dim f(Z).
\end{align*}
Note that $e_f(Z) > 0$ if and only if $Z$ drops dimension under $f$.
\end{definition}

The index is not well-defined on numerical classes of cycles; however, the following proposition shows that the index does provide a well-defined lower-bound across effective decompositions. It also rules out certain cycles from appearing in an effective decomposition.

\begin{proposition}[{\cite[Proposition 2.1]{chencoskun2015}}]\label{prop:indexbound}
Let $f:X\to Y$ be a morphism between projective varieties and let $k > m \geq 0$ be two integers. Let $Z \subset X$ be $k$-dimensional, and suppose $e_f(Z) \geq k-m > 0$. If $\DD[Z] = \sum_s a_s[E_s]$ is an effective decomposition of $Z$, then $e_f(E_s) \geq k-m$ for every $s$. \hfill $\square$
\end{proposition}

The primary topics of interest in this paper are moduli spaces of curves and their boundary strata. As discussed in Section \ref{sec:curves}, boundary strata can be realized as products of smaller moduli spaces of curves. Thus the following result is of crucial importance.

\begin{lemma}[cf. {\cite[Corollary 2.4]{chencoskun2015}}]\label{lem:cartesian}
Let $X$ and $Y$ be projective varieties, either smooth or $\mathbb{Q}$-factorial, and let $Z\subset X$ such that $[Z]$ is extremal in $\emph{Eff}_d(X)$. Then $[Z\times Y]$ is extremal in $\emph{Eff}_d(X \times Y)$.
\end{lemma}
\begin{proof}
Let
\begin{align}\label{eq:cartesiandecomp}
[Z \times Y] = \sum_s a_s [E_s] \in N^k(X\times Y)
\end{align}
be an effective decomposition. Let $\pi:X\times Y \to X$ the projection morphism; by Proposition \ref{prop:indexbound}, $e_\pi(E_s) \geq e_\pi(Z\times Y) = \dim Y$ for every $s$. Moreover, since $\pi$ is projection, $e_\pi(E_s) = \dim Y$ and $E_s = F_s \times Y$ for some $F_s \subset X$. By the projection formula, \eqref{eq:cartesiandecomp} becomes
\begin{align*}
[Z] = \sum_s a_s[F_s].
\end{align*}
Since $[Z]$ is assumed to be extremal in $X$, each $[F_s]$ is proportional to $[Z]$.

Now fix an $F_s$ and assume without loss of generality that the proportionality constant between $[F_s]$ and $[Z]$ is $1$. Suppose that the class $[F_s \times Y]$ is not proportional to $[Z\times Y]$ in $N^k(X\times Y)$. Then there must exist some $G\subset X\times Y$ such that $G$ is $k$-dimensional and
\begin{align*}
\int_{X \times Y} [Z \times Y] \cdot [G] \neq \int_{X \times Y} [F_s \times Y] \cdot [G].
\end{align*}
However, after another application of the projection formula, we find that
\begin{align*}
\int_{X} [Z] \cdot \pi_*[G] \neq \int_{X} [F_s] \cdot \pi_*[G],
\end{align*}
a contradiction.
\end{proof}

The next lemma allows the transfer of rigidity and extremality from one space to another under ideal conditions. We caution that the injectivity hypothesis is a strong requirement that is typically not satisfied, and its satisfaction depends intricately on the geometries of the spaces involved.

\begin{lemma}\label{lem:maintool}
Let $\gamma:Y \to X$ be a morphism of projective varieties which are at worst $\mathbb{Q}$-factorial, and let $Z\subset Y$ be an irreducible subvariety. Assume that $\gamma_*:N_d(Y) \to N_d(X)$ is injective. Suppose that for any effective decomposition
\begin{align*}
[\gamma(Z)] = \sum_s a_s[E_s] \in N_d(X),
\end{align*}
we have $E_s \subset \gamma(Y)$ for all $s$. If $[Z]$ is extremal (resp. rigid and extremal) in $\emph{Eff}_d(Y)$, then $[\gamma(Z)]$ is extremal (resp. rigid and extremal) in $\emph{Eff}_d(X)$.
\end{lemma}
\begin{proof}
This is a mild alteration of \cite[Proposition 2.5]{chencoskun2015}, and our proof mirrors the one given there.

Suppose
\begin{align*}
[\gamma(Z)] = \sum_{s} a_s[E_s] \in N_d(X)
\end{align*}
is an effective decomposition. Since $\gamma_*$ is injective, we have an effective decomposition
\begin{align*}
[Z] = \sum_{s} a_s[E_s'] \in N_d(Y)
\end{align*}
where $\gamma_*[E_s'] = [E_s]$. But $Z$ is extremal in $Y$, so each $[E_s']$ is proportional to $[Z]$. Again, since $\gamma_*$ is injective, $\gamma_*[E_s'] = [E_s]$ is proportional to $\gamma_*[Z] = [\gamma(Z)]$, and $[\gamma(Z)]$ is extremal.

Suppose $[\gamma(Z)]$ is not rigid. Since it is extremal it can be written
\begin{align*}
[\gamma(Z)] = c[V]
\end{align*}
for $c >0$ and some $V$ not supported on $\gamma(Z)$. A parallel argument to that just given provides a contradiction.
\end{proof}

\section{Background on Moduli Spaces of Curves}\label{sec:curves}

We collect here some of the background information necessary for this paper concerning the moduli space of curves. For a more thorough introduction to and treatment of this important space, we recommend any of \cite{harrismorrison,vakil08,acgh2013}.

Denote by $\Moduli_{g,n} = \Moduli_{g,\{p_1,\dots,p_n\}}$ the moduli space of isomorphism classes of Deligne-Mumford stable genus $g$ curves with $n$ (ordered) marked points. For fixed $g$, we may vary $n$ to obtain a family of moduli spaces related by \emphbf{forgetful morphisms}: for each $1\leq i \leq n$, the map $\pi_{i}:\Moduli_{g,n}\to\Moduli_{g,n-1}$ forgets the $i$th marked point and stabilizes the curve if necessary. The map $\pi_{n+1}$ realizes $\Moduli_{g,n+1}$ as the universal curve over $\Moduli_{g,n}$. If $S = \{p_{i_1},\dots,p_{i_m}\}\subseteq \{p_1,\dots,p_n\}$, define $\pi_S := \pi_{i_1} \circ \cdots \circ \pi_{i_m}$.

The boundary $\Moduli_{g,n} \backslash \moduli_{g,n}$ consists of nodal curves and may be written as a union of irreducible \emphbf{boundary strata}, each of which we denote by a $\Delta$ symbol, and which, depending on how we would like to emphasize the stratum, is decorated with an integer-set pair $(h;S)$ or a stable graph $\Gamma$. 
When the stratum is a \emphbf{boundary divisor}, we decorate with an integer-set pair: for $0\leq h \leq g$ and $S\subseteq \{p_1,\dots,p_n\}$, the general point of $\Delta_{h;S}$ parametrizes a genus $h$ curve containing the marked points labeled by $S$, attached at a node to a genus $g-h$ curve containing the marked points in $\{p_1,\dots,p_n\}\backslash S$.

Associated to any boundary stratum is a \emphbf{dual graph} $\Gamma$, realized as the dual graph of the general point of the stratum; such a stratum is denoted $\Delta_\Gamma$. The dual graph $\Gamma$ of $\Delta_\Gamma$ is defined as follows: if $(C;p_1\dots,p_n)$ is the marked curve parametrized by the generic point of $\Delta_\Gamma$, then $\Gamma$ has a vertex for every irreducible component of $C$ labeled by geometric genus, an edge connecting vertices when corresponding components of $C$ share a node, and labeled half-edges corresponding to the marked points $p_1,\dots,p_n$. If the half-edge corresponding to $p_i$ is attached to a vertex $v$ of $\Gamma$, we write $p_i \in v$. 
One boundary stratum $\Delta_{\Gamma'}$ is contained in another $\Delta_{\Gamma}$ if and only if $\Gamma$ can be obtained from $\Gamma'$ by a series of edge contractions. Contracting an edge of $\Gamma$ corresponds to smoothing the corresponding node of the general point of $\Delta_\Gamma$. We adopt the convention that $\Moduli_{g,n} \subseteq \Moduli_{g,n}$ is a boundary stratum with dual graph given by a single vertex with $n$ half-edges. 

A boundary stratum may be canonically identified with a product of smaller moduli spaces (modulo a symmetric group). For example,
\begin{align*}
\Delta_{h;S} \cong \Moduli_{g-h,(\{p_1,\dots,p_n\}\backslash S)\cup\{\bullet\}} \times \Moduli_{h,S\cup \{\diamond\}},
\end{align*}
where the points $\bullet$ and $\diamond$ are glued together under the isomorphism. The quotient may be non-trivial when there are no marked points or if the stratum is in the locus of curves with a self-node.

A boundary stratum $\Delta_\Gamma$ is of \emphbf{compact type} if $\Gamma$ is a tree. We call any boundary stratum of compact type where all of the genus is concentrated at one vertex of the dual graph a \emphbf{stratum of rational tails type}, and we call a genus $g$ vertex/component the \emphbf{root} of the stratum; the root is unique if $g\neq 0$. A non-root vertex/component which has exactly one edge/node (resp. more than one edge/node) is called \emphbf{external} (resp. \emphbf{internal}). A \emphbf{tail} is a tree of vertices corresponding to $\mathbb{P}^1$s with an edge connected to the root, such that the deletion of this tree does not change the connectivity of the associated dual graph.

Some specializations of Conjecture \ref{conj:extremalboundary} are already established in the literature. Most immediate is the following standard result on the effective cone of $\Moduli_{g,n}$.
\begin{proposition}\label{prop:divisorsextreme}
All boundary divisors in $\Moduli_{g,n}$ are rigid and extremal.
\end{proposition}
\begin{proof}
This result has been well-known for quite some time. In {\cite[Proposition 3.1]{chencoskun2015}}, the authors give a straightforward proof by constructing, for each $\Delta_{h;S}$, an explicit moving curve $\calC_{h;S}$ with $[\Delta_{h,S}] \cdot [\calC_{h;S}] <0 $ and applying Lemma \ref{lem:divisoronly}.
\end{proof}

Additionally, \cite{chencoskun2015} shows extremality for pinwheel strata 
in genus zero in arbitrarily-high codimension and in genus one in codimension-two. They also show that all codimension-two strata are extremal in $\Moduli_{g}$ for $g\geq 2$, as well as extremality for several other miscellaneous boundary strata, some of compact type and some contained in the irreducible divisor. In \cite{schaffler2015}, the author extends the genus zero result to allow for further special degenerations.

The combinatorial stratification of $\Moduli_{g,n}$ and Proposition \ref{prop:divisorsextreme} allow us to bootstrap rigidity and extremality of boundary divisors to rigidity and extremality for higher-codimension boundary strata relative to lower-codimension boundary strata. 

\begin{corollary}\label{cor:divisorspluscartesian}
Let $\Delta$ be a boundary stratum in $\Moduli_{g,n}$ of compact type. Let $\Delta'$ be a codimension-1 degeneration of $\Delta$ which is also of compact type. Then $[\Delta']$ is extremal in $\emph{Eff}^1(\Delta)$.
\end{corollary}
\begin{proof}
Suppose $\Delta'$ is a codimension-one degeneration of $\Delta$. We may write
\begin{align*}
\Delta &= \prod_{i=1}^r \Moduli_{g_i,n_i} = \Moduli_{g_1,n_1} \times \prod_{i=2}^r \Moduli_{g_i,n_i}
\end{align*}
and without loss of generality
\begin{align*}
\Delta' = \Moduli_{g_a,n_a} \times \Moduli_{g_b,n_b} \times \prod_{i=2}^r \Moduli_{g_i,n_i}.
\end{align*}
We view $\Moduli_{g_a,n_a} \times \Moduli_{g_b,n_b}$ as a boundary divisor in $\Moduli_{g_1,n_1}$, which is extremal by Proposition \ref{prop:divisorsextreme}. Therefore $[\Delta']$ is extremal in $\text{Eff}^1(\Delta)$ by Lemma \ref{lem:cartesian}.
\end{proof}

It is sometimes convenient to do the same for rigidity. This is not always possible: because all points are rationally equivalent on $\Moduli_{0,4} \cong \mathbb{P}^1$, pushing forward the relation $[\Delta_{0;\{i_1,i_2\}}] = [\Delta_{0;\{i_3,i_4\}}]$ on $\Moduli_{0,4}$ under a gluing morphism induces an equivalence of classes among certain classes of rational tails strata in $\Moduli_{g,n}$. These are called the Witten-Dijkgraaf-Verlinde-Verlinde (WDVV) relations, and they preclude rigidity when a boundary stratum has a dual graph with two or more adjacent trivalent vertices. However, if the dual graph does not have adjacent trivalent vertices, then Proposition \ref{prop:divisorsextreme} can be used to show both rigidity and extremality of the associated stratum simultaneously.

\begin{corollary}\label{cor:relativedivisorrigex}
Let $\Delta_{\Gamma'}$ be a boundary stratum in $\Moduli_{g,n}$ of compact type with dual graph $\Gamma'$, and fix two adjacent vertices $v,w$ of $\Gamma'$ of respective genus $g_v$ and $g_w$ and respective valence $n_v$ and $n_w$. Let $\Gamma$ be the graph obtained from $\Gamma'$ by contracting the edge between $v$ and $w$. If $g_v+g_w + n_v + n_w > 6$, 
then $[\Delta_{\Gamma'}]$ is rigid and extremal in $\emph{Eff}^1(\Delta_{\Gamma})$.
\end{corollary}
\begin{proof}
The strata may be written
\begin{align*}
\Delta_{\Gamma'} &= \Moduli_{g_v,n_v} \times \Moduli_{g_w,n_w} \times \prod_{j=1}^r \Moduli_{g_j,n_j}, \\
\Delta_{\Gamma} &= \Moduli_{g_v+g_w,n_v+n_w-2} \times \prod_{j=1}^r \Moduli_{g_j,n_j},
\end{align*}
Without loss of generality, suppose $g_v >0$ or $n_v > 3$. Fix a general point $(C_v;p_{i_1},\dots,p_{i_{n_v-1}}), \in \Moduli_{g_v,n_v-1}$ and a general point $(C_w\cup C_1 \cup \cdots \cup C_r; p_{i_1},\dots,p_{i_{s}}) \in \Moduli_{g_w,n_w-1} \times \prod_{j=1}^r \Moduli_{g_j,n_j}$. Let $\calC$ be the moving curve in $\Moduli_{g_v,n_v} \times \Moduli_{g_w,n_w} \times \prod_{j=1}^r \Moduli_{g_j,n_j}$ obtained by gluing a fixed smooth point $\diamond$ on $C_w$ (distinct from the $p_i$) to a varying point $\bullet$ on $C_v$.

If $\text{pr}:\Delta_{\Gamma} \to \Moduli_{g_v+g_w,n_v+n_w-2}$ is the projection morphism to the first factor, then by the projection formula,
\begin{align*}
[\Delta_{\Gamma'}] \cdot [\calC] &= \text{pr}^*[\Delta_{g_v;\{p_{i_1},\dots,p_{i_{n_v-1}},\bullet\}}] \cdot [\calC] \\
&= [\Delta_{g_v;\{p_{i_1},\dots,p_{i_{n_v-1}},\bullet\}}] \cdot \text{pr}_*[\calC].
\end{align*}
But $\pi_*[\calC]$ is exactly the moving curve constructed in {\cite[Proposition 3.1]{chencoskun2015}} to give negative intersection with $[\Delta_{g_v;\{p_{i_1},\dots,p_{i_{n_v-1}},\bullet\}}]$. Therefore $[\Delta_{\Gamma'}]$ is rigid and extremal in $\text{Eff}^1(\Delta_{\Gamma})$ by Lemma \ref{lem:divisoronly}.
\end{proof}

\begin{remark}\label{rem:compacttype}
The requirement that the stratum be of compact type is imposed to avoid concern over the need to write them as quotients of products of moduli spaces under symmetric group actions. One practical effect of these requirements is in ruling out degenerations that induce self-nodes.
\end{remark}

The full moduli space $\Moduli_{g,n}$ is not merely a product of smaller moduli spaces, so it remains to push extremality (and rigidity) within a boundary divisor to extremality (and rigidity) within the full space. Lemma \ref{lem:maintool} is one tool to do so, in conjunction with the Hassett spaces introduced in \cite{hassett}.

\begin{definition}
Fix \emphbf{(ordered) weight data} $\calA = (a_1,a_2,\dots,a_n)$ so that $a_i\in (0,1]\cap\mathbb{Q}$
. A (nodal) marked curve $(C;p_1,\dots,p_n)$ is \emphbf{$\calA$-stable} if 
\begin{itemize}
\item $p_i\in C$ is smooth for every $i\in [n]$;
\item $\DD \omega_C + \sum_{i=1}^n a_ip_n$ is ample; and
\item for every point $x\in C$, we have $\DD\sum_{p_i=x}a_i \leq 1$.
\end{itemize}
The \emphbf{Hassett space} $\Moduli_{g,\calA}$ is the moduli space of $\calA$-stable curves of genus $g$ up to isomorphism.
\end{definition}

When $2g-2+\sum a_i > 0$, the Hassett space $\Moduli_{g,\calA}$ is a non-empty, smooth, proper Deligne-Mumford stack. In this case there exists a \emphbf{reduction morphism}
\begin{align*}
\rho:\Moduli_{g,n}\to\Moduli_{g,\calA},
\end{align*}
which on the level of curves reduces the weights and stabilizes if necessary by contracting unstable rational components. The exceptional locus of $\rho$ consists of curves that have a rational tail with at least three marked points of total weight at most one.

\section{Main Results}\label{sec:rattail}

The aim of this section is to make progress on Conjecture \ref{conj:extremalboundary}, with a focus on rational tails strata. First, with an eye towards applying Lemma \ref{lem:maintool}, we verify the injectivity of the pushforward of the inclusion of a rational tails boundary stratum into $\Moduli_{g,n}$ using results from \cite{keel92} and \cite{tavakol}.

\begin{lemma}\label{lem:timespt}
Let $S \subset \{p_1,\dots,p_n\}$ such that $|S| \geq 2$, and let $\gamma:\Delta_{0;S} \to \Moduli_{g,n}$ be inclusion. A class $\alpha \in N^*(\Moduli_{g,n-|S|+1})$ is non-zero if and only if $\gamma_*(\alpha \otimes [pt]) \in N^*(\Moduli_{g,n})$ is non-zero.
\end{lemma}
\begin{proof}
The composition of $\gamma$ with the forgetful morphism which forgets the marks in $S$ is the identity on $\Moduli_{g,n-|S|+1}$.
\end{proof}

\begin{proposition}\label{prop:gammainjective}
The pushforward $\gamma_*:N_*(\Delta_{0;S}) \to N_*(\Moduli_{g,n})$ is injective.
\end{proposition}
\begin{proof}
Let $Q \in N^k(\Delta_{0;S})$. By \cite{keel92},
\begin{align*}
N^k(\Delta_{0;S}) \cong \bigoplus_{i=0}^k N^{k-i}(\Moduli_{g,n-|S|+1}) \otimes N^{i}(\Moduli_{0,s+1}),
\end{align*}
so we may write
\begin{align*}
Q = \sum_{i=0}^k \sum_{\ell=1}^{m_i} \alpha^{k-i}_{\ell} \otimes \beta^{i}_{\ell},
\end{align*}
where $\alpha^{k-i}_{\ell} \in N^{k-i}(\Moduli_{g,n-|S|+1})$, $\beta^{k-i}_{\ell} \in N^{i}(\Moduli_{0,|S|+1})$, and the set $\{\beta^i_1,\dots,\beta^i_{m_i}\}$ is linearly independent for each $i$. Since $N^*(\Moduli_{0,|S|+1})$ is generated by boundary divisors, and since in $\Moduli_{0,|S|+1}$ the product of two boundary strata is again a sum of boundary strata, all of the  $\beta^i_{\ell}$ can be written as sums of boundary strata. 

Let $r \in \{0,1,\dots,k\}$ be the largest index for which some $\beta^r_{\ell}$ appears in $Q$, and fix $\beta=\beta^r_{j}$ for some $j\in\{1,\dots,m_r\}$ with $\alpha_j^{k-r} \neq 0$. By \cite{tavakol}, there exists an element $\beta^\vee \in N^{|S|-2-r}(\Moduli_{0,|S|+1})$ so that $\beta \cdot \beta^\vee = [pt]$ and $\beta_{\ell}^i \cdot \beta^\vee = 0$ for all $(i,\ell) \neq (r,j)$ with $i \leq r$. Moreover, $\beta^\vee$ can also be written as a sum of boundary strata $\beta^\vee = \sum_{b} [\Delta_b]$. For each $b$, if $\Gamma_b$ is the dual graph of $\Delta_b$, construct a new graph $\tilde{\Gamma}_b$ by replacing the vertex carrying the marked point $\bullet$ with a genus $g$ vertex carrying the marks $\{p_i \stc i\not\in S\}$. Define $\tilde{\Delta}_b$ to be the boundary stratum in $\Moduli_{g,n}$ with dual graph $\tilde{\Gamma}_b$, so that $\gamma^*([\tilde{\Delta}_b]) = [\Moduli_{g,n-|S|+1}] \otimes [\Delta_b]$, and let $\tilde{\beta}^\vee = \sum_{b} [\tilde{\Delta}_b]$.

Now suppose that $\gamma_*(Q) = 0$. Then
\begin{align*}
0 &= \gamma_*(Q) \cdot \tilde{\beta}^\vee \\
&= \gamma_*\left(\left(\sum_{i=0}^k \sum_{\ell=1}^{m_i} \alpha^{k-i}_{\ell} \otimes \beta^{i}_{\ell} \right) \cdot \gamma^*\left(\sum_{b} [\tilde{\Delta}_b]\right) \right) \\
&= \gamma_*\left(\sum_{i=0}^k \sum_{\ell=1}^{m_i} \alpha^{k-i}_{\ell} \otimes \left(\beta^{i}_{\ell} \cdot \beta^\vee \right) \right) \\
&= \gamma_*\left(\sum_{\ell=1}^{m_r} \alpha^{k-r}_{\ell} \otimes \left(\beta^{r}_{\ell} \cdot \beta^\vee \right) \right) \\
&= \gamma_*\left(\alpha^{k-r}_{j} \otimes [pt] \right).
\end{align*}
By Lemma \ref{lem:timespt}, $\alpha^{k-r}_{j} = 0$ in $N^{k-r}(\Moduli_{g,n-|S|+1})$. But $\alpha_{j}^{k-r}$ was assumed to be nonzero, a contradiction. Hence $\gamma_*$ is injective.
\end{proof}

\begin{corollary}\label{cor:gengammainjective}
Let $\Delta'$ be a rational tails boundary stratum in $\Moduli_{g,n}$ contained in some other boundary stratum $\Delta$. The pushforward $\gamma_*:N_*(\Delta') \to N_*(\Delta)$ is injective.
\end{corollary}
\begin{proof}
First assume $\codim_{\Moduli_{g,n}}(\Delta') = \codim_{\Moduli_{g,n}}(\Delta)+1$. If the roots of the curves parametrized by the general points of $\Delta$ and $\Delta'$ share the same set of marks, the strata may be written
\begin{align*}
\Delta' &= \Moduli_{g,\ell} \times \Moduli_{0,n_0} \times \Moduli_{0,n_1} \times \prod_{j=2}^r \Moduli_{0,n_j}, \\
\Delta &= \Moduli_{g,\ell} \times \Moduli_{0,n_0 + n_1 - 2} \times \prod_{j=2}^r \Moduli_{0,n_j}.
\end{align*}
Then by \cite{keel92},
\begin{align*}
N^*(\Delta') & \cong N^*(\Moduli_{g,\ell}) \otimes N^*\left(\Moduli_{0,n_0} \times \Moduli_{0,n_1}\right) \otimes N^*\left(\prod_{j=2}^r \Moduli_{0,n_j}\right); \\
N^*(\Delta) &\cong N^*(\Moduli_{g,\ell}) \otimes N^*\left(\Moduli_{0,n_0 + n_1 - 2}\right) \otimes N^*\left(\prod_{j=2}^r \Moduli_{0,n_j}\right).
\end{align*}
By Proposition \ref{prop:gammainjective}, $\eta:N^*\left(\Moduli_{0,n_0} \times \Moduli_{0,n_1}\right) \to N^*\left(\Moduli_{0,n_0 + n_1 - 2}\right)$ is injective, so $\gamma_* = \text{id} \otimes \eta \otimes \text{id}$ is injective.

If the roots do not share the same set of marks, the strata may be written
\begin{align*}
\Delta' &= \Moduli_{g,\ell-m} \times \Moduli_{0,m+2} \times \prod_{j=1}^r \Moduli_{0,n_j}, \\
\Delta &= \Moduli_{g,\ell} \times \prod_{j=1}^r \Moduli_{0,n_j},
\end{align*}
and as above, Proposition \ref{prop:gammainjective} implies that $\eta:N^*\left(\Moduli_{g,\ell-m} \times \Moduli_{0,m+2}\right) \to N^*\left(\Moduli_{g,\ell}\right)$ is injective. Hence $\gamma_* = \eta \otimes \text{id}$ is injective.

If $\codim_{\Moduli_{g,n}}(\Delta') > \codim_{\Moduli_{g,n}}(\Delta) + 1$, then $\gamma_*$ is the composition of injective functions and is therefore injective.
\end{proof}

For rational curves, \cite[Lifting Lemma]{schaffler2015} establishes that for the gluing morphism 
\begin{align} \label{eq:iota}
\iota : \Moduli_{0,n+1} \times \Moduli_{0,\{q_{0}, q_1,q_2\}} \to \Moduli_{0,n+2}
\end{align}
which attaches $p_{n+1}$ to $q_0$, if $\alpha \in \text{Eff}^k(\Moduli_{0,n+1})$ is extremal, then $\iota_*(\alpha, [\Moduli_{0,\{q_1,q_2,q_3\}}]) \in \text{Eff}^{k+1}(\Moduli_{0,n+2})$ is extremal as well. With Proposition \ref{prop:gammainjective}, this result can be extended, which in turn implies the extremality of all boundary strata in genus zero.

\begin{lemma}\label{lem:addmodulitail}
Let $Z \subseteq \Moduli_{g,n+1}$, fix $\ell \geq 3$, and let $\gamma$ be the gluing morphism
\begin{align*}
\gamma: \Moduli_{g,n+1} \times \Moduli_{0,\ell+1} \to \Moduli_{g,n+\ell}.
\end{align*}
If $[Z] \in \emph{Eff}^k(\Moduli_{g,n+1})$ is extremal (resp. rigid and extremal), then $[\gamma(Z \times \Moduli_{0,\ell+1})] \in \emph{Eff}^{k+1}(\Moduli_{g,n+\ell})$ is extremal (resp. rigid and extremal).
\end{lemma}
\begin{proof}
Label the first $n+1$ points $p_1,\dots,p_n, \bullet$ and the last $\ell+1$ points $\diamond, p_{n+1},\dots,p_{n+\ell}$, so that $\gamma$ attaches $\bullet$ to $\diamond$. Define weight data $\calA = (a_1,\dots,a_{n+\ell})$ by
\begin{align*}
a_i= \left\{ \begin{array}{cl}
1, & i \leq n \\
\frac{1}{\ell}, & i > n \end{array} \right.
\end{align*}
and let $\rho : \Moduli_{g,n+\ell} \to \Moduli_{g,\calA}$ be the reduction morphism. The exceptional locus of $\rho$ is
\begin{align*}
\bigcup_{\substack{L \subseteq \{p_{n+1},\dots,p_{n+\ell}\} \\ |L|>2}} \Delta_{0;L},
\end{align*}
and note that $e_{\rho}(\Delta_{0;L}) = |L|-2$ for any $L \subseteq \{p_{n+1},\dots,p_{n+\ell}\}$ with $|L|>2$. Let
\begin{align*}
[\gamma(Z\times \Moduli_{0,\ell+1})] &= \sum_{j=1}^r a_j[E_j] \in N^{k+1}(\Moduli_{g,n+\ell})
\end{align*}
be an effective decomposition. The dimension of $\gamma(Z\times \Moduli_{0,\ell+1})$ is $\dim(Z) + \ell - 2$, and the dimension of $\rho(\gamma(Z\times \Moduli_{0,\ell+1}))$ is $\dim(Z)$; therefore $e_{\rho}(\gamma(Z\times \Moduli_{0,\ell+1})) = \ell-2$. By Proposition \ref{prop:indexbound}, $e_{\rho}(E_j) \geq e_{\rho}(\gamma(Z \times \Moduli_{0,\ell+1})) = \ell - 2 > 0$ for any $j$, so $E_j \subseteq \Delta_{0;\{p_{n+1},\dots,p_{\ell}\}} = \gamma(\Moduli_{g,n+1} \times \Moduli_{0,\ell+1})$. Since $\gamma_*$ is injective by Proposition \ref{prop:gammainjective} and $[Z \times \Moduli_{0,\ell+1}]$ is extremal (resp. rigid and extremal) in $\text{Eff}^{k}(\Moduli_{g,n+1} \times \Moduli_{0,\ell+1})$ by Lemma \ref{lem:cartesian}, we conclude that $[\gamma(Z\times \Moduli_{0,\ell+1})]\in \text{Eff}^{k+1}(\Moduli_{g,n+\ell})$ is extremal (resp. rigid and extremal) by Lemma \ref{lem:maintool}.
\end{proof}

\begin{corollary}\label{cor:m0nextremal}
All boundary strata in $\Moduli_{0,n}$ are extremal.
\end{corollary}
\begin{proof}
Any boundary stratum in $\Moduli_{0,n}$ is either a boundary divisor, which is therefore extremal by Proposition \ref{prop:divisorsextreme}, or can be written as the image of some sequence of gluing morphisms, which is therefore extremal by the extremality of boundary divisors, \cite[Lifting Lemma]{schaffler2015}, and Lemma \ref{lem:addmodulitail}.
\end{proof}

\begin{remark}\label{rem:counterexamplem0n}
The analogous statement of Corollary \ref{cor:m0nextremal} fails for $\tilde{\moduli}_{0,n} = \Moduli_{0,n} / \Sigma_n$, the moduli space of stable rational curves with $n$ \emph{unordered} marked points. A counterexample for $n=7$ is computed explicitly in \cite[Table 1]{moon17}: the $F$-curve class $[F_{1,1,2,3}]$ is not extremal in $\text{Eff}^3(\tilde{\moduli}_{0,n})$, as $[F_{1,1,2,3}] = \frac{1}{2}([F_{1,1,1,4}] + [F_{1,2,2,2}])$.
\end{remark}

\begin{corollary}\label{cor:allmoduli}
Let $\Delta \subset \Moduli_{g,n}$ be a rational tails boundary stratum of codimension $k$ such that the dual graph of $\Delta$ has no trivalent non-root vertices. Then $[\Delta]$ is extremal in $\emph{Eff}^k(\Moduli_{g,n})$.
\end{corollary}
\begin{proof}
The proof is identical to Corollary \ref{cor:m0nextremal}, except that the gluing morphisms from \cite[Lifting Lemma]{schaffler2015} do not appear.
\end{proof}

Unlike \cite[Lifting Lemma]{schaffler2015}, Lemma \ref{lem:addmodulitail} is independent of genus.
However, the assumption of $\ell \geq 3$ is crucial,
since it guarantees that $e_\rho(Z \times \Moduli_{0,\ell+1})$
is positive, and the lack of direct generalization of \cite[Lifting Lemma]{schaffler2015} to positive genus is the major hurdle in extending Corollary \ref{cor:m0nextremal} to an analogous statement regarding rational tails strata for $g\geq 1$. There are some special cases in which the lack of a fully generalized Lifting Lemma can be circumvented; one technique, which first appeared in \cite{chentarasca} in regards to hyperelliptic classes, is to push forward an effective decomposition of a class of interest onto rigid and extremal classes via forgetful morphisms in enough ways so that a modularity argument implies that the decomposition must be trivial. Although the rigidity requirement on the image of the forgetful morphisms fails to hold in some cases, in certain situations, composing forgetful morphisms can overcome this obstacle; the next lemma addresses this situation.

Similar to the morphism $\iota$ in \eqref{eq:iota}, define special gluing morphisms as follows. Fix $n\geq 1$ and let $S = \{q_1,\dots,q_s\}$ be an ordered finite set with $|S| \geq 2$. For $p_i \in \{p_1,\dots,p_n\}$, define
\begin{align*}
\tau^{p_i,S}:\Moduli_{g,n} \to \Moduli_{g,n+s-1}
\end{align*}
as the restriction to the first factor of the morphism
\begin{align*}
\Moduli_{g,n} \times \prod_{j=1}^{s-2} \Moduli_{0,\{\diamond_{j-1},q_j, \bullet_j\}} \times \Moduli_{0,\{\diamond_{s-2},q_{s-1},q_{s}\}} \to \Moduli_{g,n+s-1}
\end{align*}
which glues $p_i$ to $\diamond_0$ and $\bullet_j$ to $\diamond_j$ for $1\leq j \leq s-2$.

\begin{lemma}\label{lem:nomodulireduction}
Fix $m \geq 2$, and let $\Delta \subset \Moduli_{g,m+\ell}$ be a boundary stratum of codimension $k$. For $1\leq j \leq m$, let $T_j$ be an ordered set with $|T_j| \geq 2$. Suppose for for every $1 \leq j \leq m$ and every $S \subseteq T_j$ with $|S| = 2$, the class $\tau^{p_{j},S}_*[\Delta] \in \emph{Eff}^{k+1}(\Moduli_{g,m+\ell+1})$ is rigid and extremal. Set $K = k+|T_1| + \cdots + |T_m| - m$ and $n = \ell + |T_1| + \cdots |T_m|$, and define 
\begin{align*}
\tau := \tau^{p_{1},T_1} \circ \cdots \circ \tau^{p_{m},T_m}.
\end{align*}
Then the class $\tau_*[\Delta] \in \emph{Eff}^{K}(\Moduli_{g,n})$ is extremal.
\end{lemma}
\begin{proof}
For each $T_j$, and for each $P \subseteq T_j$, let $\varpi_P$ be the composition of forgetful morphisms which forgets all of the points in $T_j \backslash P$. Note that $|T_j| \geq 2$ for each $j$.

Fix a two-element subset $\{x_1,x_2\} \subseteq T_1$, and fix singletons $y_{j} \in T_j$ for $j = 2, \dots, m$. Define
\begin{align*}
\varpi := \varpi_{\{x_1,x_2\}} \circ \varpi_{\{y_2\}} \circ \varpi_{\{y_3\}} \circ \cdots \circ \varpi_{\{y_m\}},
\end{align*}
and let
\begin{align}\label{eq:nomodulieffdecomp}
\tau_*[\Delta] = \sum_s a_s[E_s] + \sum_t b_t[F_t] \in N^K(\Moduli_{g,n})
\end{align}
be an effective decomposition such that $\varpi_{*}[E_s] \neq 0$ for all $s$ and $\varpi_{*}[F_t] = 0$ for all $t$. Fix one of the $[E_s]$; by definition of $\varpi$, it must survive pushforward by $\pi_{q}$ for $q \in \left(\bigcup_{j=1}^m T_j \right) \backslash \{x_1,x_2,y_1,\dots,y_m\}$. Since $\varpi_{*}\tau_*[\Delta] = \tau^{p_1,\{x_1,x_2\}}_*[\Delta]$ is extremal in $\text{Eff}^{k+1}(\Moduli_{g,m+\ell+1})$ by assumption, the class $\varpi_{*}[E_s]$ is a positive multiple of $\tau^{p_1,\{x_1,x_2\}}_*[\Delta]$; therefore $[E_s]$ survives pushforward by $\pi_{x_1}$ and $\pi_{x_2}$ as well. Exchanging the roles of $T_1$ and $T_2$ implies that $[E_s]$ also survives pushforward by
\begin{align*}
\varpi' := \varpi_{\{y_2,z\}} \circ \varpi_{\{x_1\}} \circ \varpi_{\{y_3\}} \circ \cdots \circ \varpi_{\{y_m\}},
\end{align*}
where $z \in T_2\backslash\{y_2\}$. Again since $\varpi'_{*}[\Delta] = \tau^{p_2,\{y_2,z\}}_*[\Delta]$ is extremal in $\text{Eff}^{k+1}(\Moduli_{g,m+\ell+1})$, the class $\varpi'_{*}[E_s]$ is a positive multiple of $\tau^{p_2,\{y_2,z\}}_*[\Delta]$, and hence $[E_s]$ also survives pushforward by $\pi_{y_2}$. By iterating this argument, we see that $[E_s]$ survives pushforward by $\pi_q$ for all $q \in \bigcup_{j=1}^m T_j$.

Since each $\tau^{p_{j},\{x_1,x_2\}}_*[\Delta] \in \text{Eff}^{k+1}(\Moduli_{g,m+\ell+1})$ is rigid and extremal,
\begin{align*}
E_s \subseteq \bigcap_{\substack{j=1,\dots,m \\ x_1,x_2 \in T_j \text{ distinct} \\ y_k \in T_{k \neq j}}} \varpi_{\bar{x},\bar{y}}^{-1} \left(\tau^{p_{j},\{x_1,x_2\}}(\Delta)\right),
\end{align*}
where $\varpi_{\bar{x},\bar{y}} = \varpi_{\{x_1,x_2\}} \circ \varpi_{\{y_1\}} \circ \varpi_{\{y_2\}} \circ \cdots \circ \varpi_{\{y_{j-1}\}} \circ \varpi_{\{y_{j+1}\}} \circ \cdots \circ \varpi_{\{y_m\}}$. This means that the curve parametrized by the general point of $E_s$ must have a dual graph obtained from the dual graph of $\Delta$ by replacing the half edges $p_1,\dots,p_m$ by $m$ rational chains $C_1,\dots,C_m$, and for each $j$, the points in $T_j$ must be on the same chain.
If at least one of the vertices of some $C_j$ has valence greater than three, then $\varpi_{\{q\}*}[E_s] = 0$ for some $q \in T_j$, a contradiction. Thus the vertices in $C_1,\dots,C_m$ must all be trivalent. For any such distribution, $[E_s]$ is equivalent to $\tau_*[\Delta]$ via the WDVV relations.

After subtracting and rescaling, \eqref{eq:nomodulieffdecomp} becomes
\begin{align*}
\tau_*[\Delta] &= \sum_t \hat{b}_t[F_t] \in N^K(\Moduli_{g,n}).
\end{align*}
However, $\varpi_{*}\tau_*[\Delta] \neq 0 = \varpi_{*}[F_t]$, a contradiction; therefore, only trivial effective decompositions of $\tau_*[\Delta]$ exist, and $\tau_*[\Delta]$ is extremal.
\end{proof}

Corollary \ref{cor:allmoduli} shows extremality for rational tails strata which have moduli coming from all non-root vertices of their dual graphs; the following corollary does the same for rational tails strata which have moduli coming from none of the non-root vertices.

\begin{corollary}\label{cor:nomoduli}
Let $\Delta \subset \Moduli_{g,n}$ be a rational tails boundary stratum of codimension $k$ such that every non-root vertex of the dual graph of $\Delta$ is trivalent. Then $[\Delta]$ is extremal in $\emph{Eff}^k(\Moduli_{g,n})$.
\end{corollary}
\begin{proof}
If $\Delta = \Delta_{0;\{p_{i_1},p_{i_2}\}}$, then $[\Delta]$ is extremal by Proposition \ref{prop:divisorsextreme}. Otherwise, suppose the dual graph of $\Delta$ has $m\geq 2$ tails. For $1\leq j \leq m$, let $T_j$ be the labeled half edges in the $m$th tail. Then $\tau^{p_j,\{x_1,x_2\}}_*[\Moduli_{g,n}] = [\Delta_{0;\{x_1,x_2\}}]$ is rigid and extremal by Proposition \ref{prop:divisorsextreme}, and the hypotheses of Lemma \ref{lem:nomodulireduction} are satisfied.
\end{proof}

As another application, the following proposition generalizes the genus zero (arbitrary codimension) and genus one (codimension two) analogues from \cite[Theorem 6.1, 6.2, 7.2]{chencoskun2015}. First recall the definition of pinwheel strata from \cite{bcomega}: fix a partition $P_{1} \sqcup \cdots \sqcup P_r = \{p_1,\dots,p_n\}$. When $|P_i| = 1$  denote by $\bullet_i$ the element of the singleton $P_i$; for $|P_i|>1$, introduce new labels $\bullet_i$ and $\diamond_i$. The \emphbf{pinwheel stratum} associated to this partition is the image of the gluing morphism
\begin{align*}
\gamma: \Moduli_{g, \{\bullet_1, \ldots, \bullet_r\}} \times \prod_{|P_i|>1} \Moduli_{0, \{\diamond_i\}\cup P_i} \to \Moduli_{g,n}
\end{align*}
that glues together each $\bullet_i$ with $\diamond_i$.

\begin{proposition}\label{prop:pinwheel}
All pinwheel strata are extremal in $\Moduli_{g,n}$.
\end{proposition}
\begin{proof}
Induct on codimension: suppose all pinwheel strata of codimension $k$ are extremal in the relevant cones of their respective moduli spaces, and let $\Delta \subset \Moduli_{g,n}$ be a rational tails pinwheel stratum of codimension $k+1$. The base case is provided by the extremality of boundary divisors in Proposition \ref{prop:divisorsextreme}. If all external vertices of the dual graph of $\Delta$ are trivalent, then $[\Delta] \in \text{Eff}^{k+1}(\Moduli_{g,n})$ is extremal by Corollary \ref{cor:nomoduli}. Otherwise, fix one of the external vertices with valence at least four, with associated marked points $S \subset \{p_1,\dots,p_n\}$. Then $[\pi_{S\backslash\{q\}}(\Delta)]$ is extremal in $\text{Eff}^k(\Moduli_{g,n-|S|+1})$ by the induction hypothesis, so Lemma \ref{lem:addmodulitail}. implies that $[\Delta] \in \text{Eff}^{k+1}(\Moduli_{g,n})$ is extremal.
\end{proof}

\subsection*{Extensions}

Some of the above results may be pushed further in order to cover other boundary strata than those explicitly mentioned so far. In Lemma \ref{lem:addmodulitail}, the necessity of restricting to only gluing on rational tails comes from two sources: reduction morphisms to Hassett spaces only contract rational components, and the proof given for the injectivity of the pushforward of inclusion (Proposition \ref{prop:gammainjective}) does not have an analogue for strata not of rational tails type. There are partial responses to both of these concerns in isolated cases.

The role of Hassett spaces can be filled by moduli spaces of genus $g$ \emphbf{pseudostable} curves with $n$ ordered marked points $\Moduli_{g,n}^{ps}$, with the reduction morphisms replaced by the \emphbf{first divisorial contraction} of the log minimal model program $\varepsilon:\Moduli_{g,n} \to \Moduli_{g,n}^{ps}$ (see \cite{hassetthyeon,afs,afsv}), which contracts $\Delta_{1;\varnothing}$ by replacing an unmarked elliptic tail by a cusp. The injectivity of the pushforward of inclusion is known to hold for $\Delta_{1;\varnothing} \subset \Moduli_{g,n}$ for some small $g$ and $n$ (though not in a systematic way; for small genus instances, see \cite{faberchowrings1, fabercodim2, edidin1992}). Therefore, for certain boundary strata whose generic points parametrize curves with unmarked elliptic components, the setup in the proof of Lemma \ref{lem:addmodulitail} can be modified by replacing $\rho$ with $\varepsilon$.

Building on the work of \cite{fulgerlehmann2016}, some of our results may be extended to the \emphbf{pseudoeffective cone} $\overline{\text{Eff}}_d(\Moduli_{g,n})$, defined as the closure of the effective cone in the usual $\mathbb{R}^n$ topology on $N_d(\Moduli_{g,n})$. In particular, Lemma \ref{lem:maintool} may be extended to show rigidity and extremality in the pseudoeffective cone: as discussed in \cite[Remark 2.7]{chencoskun2015} and \cite[Example 4.17]{fulgerlehmann2016}, Proposition \ref{prop:indexbound} extends to the pseudoeffective cone, and the proof in Lemma \ref{lem:maintool} then carries through directly. The argument in Lemma \ref{lem:cartesian} also readily extends to the pseudoeffective cone, and thus the results which rely directly on these lemmas generalize to the pseudoeffective cone.

It is not known if the strategy of forgetting onto rigid classes carries through to the pseudoeffective setting, so results which rely on Lemma \ref{lem:nomodulireduction} cannot immediately be extended. In general there are a great many subtleties involved in moving extremality from the effective cone to the pseudoeffective cone, especially for higher codimension classes, and at present we are unable to extend the extremality of all of the strata considered here to the pseudoeffective cone without the presence of a morphism such as $\rho$ or $\varepsilon$. See \cite{fulgerlehmann2017} for further discussion and results on some of the challenges associated with moving from effective to pseudoeffective.

\bibliographystyle{amsalpha} 
\bibliography{bibliography.bib}

\end{document}